\documentclass[11pt,a4paper]{article}
\usepackage{latexsym,amsfonts,amsthm,amsmath,amssymb,cases}
\usepackage{cite}
\usepackage{enumerate}
\usepackage{graphicx}
\usepackage{color}
\usepackage{bm}

\numberwithin{equation}{section}
\newtheorem{thm}{Theorem}[section]

\newtheorem{lemma}{Lemma}[section]

\newtheorem{remark}{Remark}[section]

\newtheorem{defn}{Definition}[section]

\newcommand{\ds}{\displaystyle}
\def\mathrm{\mbox}

\numberwithin{remark}{section}

\begin{document}
%\begin{CJK}{GBK}{kai}
\title{{\Large \bf Qualitative properties of solutions to a fractional pseudo-parabolic equation with singular potential}}
\author{Xiang-kun Shao, Nan-jing Huang\footnote{Corresponding author.  E-mail addresses: nanjinghuang@hotmail.com; njhuang@scu.edu.cn}~ and Xue-song Li
\\
{\small\it  Department of Mathematics, Sichuan University, Chengdu, Sichuan 610064, P.R. China}}
\date{}
\maketitle
\begin{center}
\begin{minipage}{5.5in}
\noindent{\bf Abstract.} This paper investigates the initial boundary value problem for a fractional pseudo-parabolic equation with singular potential. The global existence and blow-up of solutions to the initial boundary value problem  are obtained at low initial energy. Moreover, the exponential decay estimates for global solutions and energy functional are further derived, and the upper and lower bounds of both blow-up time and blow-up rate for blow-up solutions are respectively estimated. Specifically, we extend the method for proving blow-up of solutions with negative initial energy in previous literatures to cases involving nonnegative initial energy, which broadens the applicability of this method. Finally, for the corresponding stationary problem, the existence of ground-state solutions is established, and it is proved that the global solutions strongly converge to the solutions of stationary problem as time tends to infinity.
\\ \ \\%stationary This work serves as a supplement to previous studies.
\end{minipage}
\end{center}

\section{Introduction}
\quad\quad
This paper is dedicated to the study of the global existence and blow-up of weak solutions for the following initial boundary value problem (IBVP) of fractional pseudo-parabolic equation (FPPE) with singular potential:
\begin{align}\label{1.1}
 \left\{\begin{array}{ll}
     \ds \frac{u_t}{|x|^{2s}}+(-\Delta)^su+(-\Delta)^su_t=|u|^{p-2}u,\quad &x\in\Omega,\ t>0,\\
     \ds u(x,t)=0,\quad &x\in\partial\Omega,\ t>0,\\
     \ds u(x,0)=u_0(x),\quad &x\in\Omega,
    \end{array}\right.
\end{align}
where $s\in(0,1)$ is a constant, $(-\Delta)^s$ is the spectral fractional Laplacian operator (FLO), $\Omega$ is a bounded subset in $\mathbb{R}^n$ ($n>2s$) and its boundary $\partial\Omega$ is smooth, and
\begin{equation}\label{p}
p\in\left(2,\frac{2n}{n-2s}\right].
\end{equation}

Since the last century, the classical pseudo-parabolic equation (PPE)
\begin{equation}\label{333}
u_t-\Delta u-k\Delta u_t=f(u)
\end{equation}
has many important applications in natural sciences, for instance, it can be used to describe the heat conduction involving conductive temperature and thermodynamic temperature\cite{chen1968theory}, the nonsteady flows of second-order fluids \cite{MR119598,MR153255}, the diffusion processes in solids and gases \cite{MR586062,MR1575106} and so on. There also have been many research achievements on \eqref{333} with polynomial nonlinearity. Specifically, when $k>0$ and $f(u)=u^q$, Cao et al. \cite{MR2523294} determined the critical global existence exponent and the critical Fujita exponent for the initial value problem (IVP) of \eqref{333} when $q>0$, and Yang et al. \cite{MR2981259} further derived the secondary critical exponent for this problem when $q>1$. For the case $k=1$ and $f(u)=u^q$ ($q>1$), Xu and Su \cite{MR3045640} proved the global existence and blow-up of solutions to the IBVP of \eqref{333}, and Luo \cite{MR3372307} established the bounds of blow-up time for blow-up solutions. When $k=1$ and $f(u)=|u|^{q-1}u$ ($q>1$), Xu and Zhou \cite{MR3745341} investigated the IBVP of \eqref{333}, providing a sufficient condition for finite time blow-up of solutions and estimating the upper bound of blow-up time.%We can refer to \cite{MR3987484,MR264231,MR330774,MR0437936,MR437936,Cao2024review} for more physical background and theoretical results concerned with \eqref{333}.

%For example, Cao et al. \cite{MR2523294} considered the unique existence and blow-up phenomenon of solutions for the initial value problem (IVP) of \eqref{333} with $f(u)=u^r$ ($r>0$). Subsequently, Xu and Su \cite{MR3045640} investigated the global existence and blow-up of solutions for the IBVP of \eqref{333} with $f(u)=u^r$ ($r\in(1,\infty)$ when $n=1,2$; $r\in(1,\frac{n+2}{n-2})$ when $n>2$).

The classical Laplace operator $\Delta$, constrained by its local characteristic, exhibits significant limitations in characterizing many complex phenomena such as anomalous diffusion and long-range correlations. The breakthrough of this limitation stems from the introduction of the FLO $(-\Delta)^s$. As an extension of the classical Laplacian, the FLO can be probabilistically interpreted as the infinitesimal generator of L\'{e}vy stationary diffusion process (see \cite{MR1406564,MR2105239}). Due to its nonlocal characteristic, the FLO plays a significant role in modeling processes with memory and genetic properties, such as image processing, signal processing, diffusion processes in physics, fluid flow on a curved surface and modeling in the financial field, see, for example,  \cite{MR3485116,MR2954615,MR4063213,MR3547446,MR2090004,MR3445279,MR3816754,MR4832638}.
%,MR2737788

For the diffusion phenomena modeled by the PPE \eqref{333}, if the historical states of the diffusing substance influence the ongoing diffusion process, it is appropriate to employ the FLO to characterize this phenomenon. Consequently, FPPE has attracted growing attention from scholars in recent years. For instance, Jin et al. \cite{MR3641746} investigated the Cauchy problem for the following IVP of FPPE:
\begin{align*}
 \left\{\begin{array}{ll}
     \ds u_t+(-\Delta)^s u-k\Delta u_t=u^{p+1},\quad &x\in\mathbb{R}^n,\ t>0,\\
     \ds u(x,0)=u_0(x),\quad &x\in \mathbb{R}^n,
    \end{array}\right.
\end{align*}
where $n\geq 1$, $k,s>0$ and $p>{4s}/n$. The authors obtained the existence of global solutions and established some decay estimates. Later on, Tuan et al. \cite{MR4347352} studied the following IBVP of FPPE:
\begin{align*}
 \left\{\begin{array}{ll}
     \ds u_t+(-\Delta)^{s_1}u+k(-\Delta)^{s_2}u_t=f(u),\quad &x\in\Omega,\ t>0,\\
     \ds u(x,t)=0,\quad &x\in\partial\Omega,\ t>0,\\
     \ds u(x,0)=u_0(x),\quad &x\in\Omega,
    \end{array}\right.
\end{align*}
and established the global existence, finite-time blow-up, and asymptotic behavior of solutions, where $s_1,s_2\in(0,1)$, $k>0$, $f$ satisfies some appropriate assumptions. For more studies on FPPE, see \cite{MR4145829,MR4828786}.

%Liu et al. \cite{MR4347352} studied the following IBVP of FPPE:
%\begin{align*}
% \left\{\begin{array}{ll}
%     \ds u_t+(-\Delta)^su+(-\Delta)^su_t=u\log|u|,\quad &x\in\Omega,\ t>0,\\
%     \ds u(x,t)=0,\quad &x\in\mathbb{R}^n\backslash \Omega,\ t>0,\\
%     \ds u(x,0)=u_0(x),\quad &x\in\Omega,
 %   \end{array}\right.
%\end{align*}
%where $\Omega\subset\mathbb{R}^n$ ($n>2s$) is a bounded domain with smooth boundary $\partial\Omega$, $s\in(0,1)$. They proved the global existence and blow-up of solutions at subcritical and critical initial energies, and provided the decay estimate for energy functional.

Additionally, the non-Newtonian filtration equation governing the motion of non-Newtonian fluids through rigid porous media may involve a spatial function $\xi(x)=|x|^{-s}~(s\geq0)$, which can be interpreted as a special medium void (see \cite{MR1881297,MR3162384,MR2036070}). %For example, Tan \cite{MR2036070} studied the IBVP for the following parabolic equation with singular potential:
%\begin{align*}
% \left\{\begin{array}{ll}
%     \ds \frac{u_{t}}{|x|^2}-\mathrm{div}(|\nabla u|^{q-2}\nabla u)=|u|^p,\quad &x\in\Omega,\ t>0,\\
%     \ds u(x,t)=0,\quad &x\in\partial\Omega,\ t>0,\\
%     \ds u(x,0)=u_0(x),~u(x,0)\geq 0,~u(x,0)\not\equiv 0, \quad &x\in\Omega,
%    \end{array}\right.
%\end{align*}
%and established the global existence, asymptotic estimates, and finite time blow-up of solutions. Here, $\Omega\subset\mathbb{R}^n$ ($n\geq 3$) is a bounded domain with smooth boundary $\partial\Omega$, $2<q<n$ and $q-1<p<\frac{nq}{n-q}-1$.
Incorporating the effect of medium void into the heat conduction or diffusion processes described by the PPE \eqref{333}, Lian et al. \cite{MR4104462} investigated the IBVP of PPE with singular potential:
\begin{align}\label{1.3-00}
 \left\{\begin{array}{ll}
     \ds \frac{u_{t}}{|x|^s}-\Delta u-\Delta u_t=|u|^{p-2}u,\quad &x\in\Omega,\ t>0,\\
     \ds u(x,t)=0,\quad &x\in\partial\Omega,\ t>0,\\
     \ds u(x,0)=u_0(x),\quad &x\in\Omega,
    \end{array}\right.
\end{align}
where $\Omega\subset\mathbb{R}^n$ ($n>2s$) is a bounded domain with smooth boundary $\partial\Omega$, $s\in[0,2]$, and $2<p<\frac{2n}{n-2}$. They proved the existence of global solutions and finite time blow-up solutions by employing the potential well theory and the concavity method, and further derived the decay estimate of global solutions and the upper bound of blow-up time.

Recently, considering both the memory effect and medium void, a very interesting work \cite{MR4180066} by Xie et al. showed the unique existence of local solutions to the IBVP \eqref{1.1} of FPPE with singular potential by applying Banach's fixed point theorem.

One important and interesting issue of studying the qualitative properties of solutions to \eqref{1.1}, after establishing the local unique existence of solutions, is the investigation of global existence and blow-up phenomena, as well as the analysis of asymptotic behavior and lifespan (i.e., blow-up time) of solutions. However, to the best of the authors' knowledge, these fundamental problems for \eqref{1.1} remain unexplored in the literature. The present paper is therefore aims to address these problems under appropriate conditions, thereby enriching the work initiated by Xie et al. \cite{MR4180066}.

We would like to mention that the energy method has previously been applied to obtain blow-up results for the IBVP of parabolic equations under negative initial energy (see \cite{MR3987484,MR3665650}). Clearly, such an energy method may not be directly applied to deal with the blow-up of solutions to \eqref{1.1} under nonnegative initial energy. Instead, one needs to establish some new estimations concerning with solutions, towards the blow-up results for \eqref{1.1}.
By applying the potential well method, the energy method, and variational method, performing tedious calculations and notably utilizing a series of differential-integral inequalities, we are able to obtain some qualitative properties of solutions to \eqref{1.1}. The main contributions of this paper lie in several points as follows: (i) the existence of global solutions to \eqref{1.1} at low initial energy is established, and the exponential decay estimates for global solutions and energy functional are obtained by utilizing the potential well method; (ii) the finite time blow-up of solutions to \eqref{1.1} at low initial energy is proved and the upper and lower bounds of both blow-up time and blow-up rate are estimated by employing the energy method, which broadens the applicability of blow-up proof techniques; (iii) for the corresponding stationary problem, the existence of ground-state solutions is established by using the variational method, and it is demonstrated that the global solutions of \eqref{1.1} strongly converge to the solutions of stationary problem as time tends to infinity.

The rest of this paper is organised as follows. In Section 2, we recall some definitions for the related space and functionals and provide some essential lemmas. In Section 3, we show our main results and their proofs.

\section{Preliminaries}
\quad\quad
In the sequel, $(\cdot,\cdot)$ represents the inner product of $L^2(\Omega)$ and $\|\cdot\|_r$ ($1\leq r\leq\infty$) represents the norm of $L^r(\Omega)$. The operator $(-\Delta)^s$ in \eqref{1.1} is the fractional Laplacian operator via a spectral definition
\begin{equation*}
(-\Delta)^s\phi=\sum\limits_{i=1}^\infty a_i\lambda_i^se_i
\end{equation*}
for all $\phi\in H_0^s(\Omega)$ as in \cite{MR3360740}, where $e_i$ and $\lambda_i$ are respectively the eigenfunctions and eigenvalues of the Laplacian operator $-\Delta$ in $\Omega$ with homogeneous Dirichlet boundary condition, while $a_i=\int_\Omega\phi e_idx$. From \cite{MR3360740,MR3023003,MR4180066}, we recall that the space
\begin{equation*}
H_0^s(\Omega)=\left\{\phi=\sum_{i=1}^\infty a_ie_i\in L^2(\Omega)\bigg|\left(\sum_{i=1}^\infty a_i^2\lambda_i^s\right)^{\frac{1}{2}}<\infty\right\},
\end{equation*}
equipped with the inner product
\begin{equation*}
(\phi,\nu)_{H_0^s}=\left((-\Delta)^{\frac{s}{2}} \phi,(-\Delta)^{\frac{s}{2}} \nu\right)
\end{equation*}
for all $\phi, \nu\in H_0^s(\Omega)$ and the norm
\begin{equation}\label{fanshudy-22}
\|\phi\|_{H_0^s}=\left\|(-\Delta)^{\frac{s}{2}}\phi\right\|_2,
\end{equation}
is a Hilbert space.

According to \cite[Theorem 6]{MR3490775} and combining the above definitions, we have the following Hardy inequality for the spectral fractional Laplacian: There exists a constant $C_*>0$ such that
\begin{equation}\label{f111-22}
\int_\Omega\frac{\phi^2}{|x|^{2s}}dx\leq C_*\left\|(-\Delta)^{\frac{s}{2}}\phi\right\|_2^2,\quad \forall\phi\in H_0^s(\Omega).
\end{equation}
The constant $C>0$, satisfying
\begin{equation}\label{C}
\|\phi\|_r\leq C\left\|(-\Delta)^{\frac{s}{2}}\phi\right\|_2,\quad \forall\phi\in H_0^s(\Omega),
\end{equation}
denotes the best embedding constant of $H_0^s(\Omega)\hookrightarrow L^r(\Omega)$ for $2\leq r\leq\frac{2n}{n-2s}$ ($n>2s$) (see \cite{MR2911424}).

For $\phi\in H_0^s(\Omega)$,  the energy functional $J(\phi)$ and the Nehari functional $I(\phi)$ are respectively defined by setting
\begin{align}
J(\phi)=\frac{1}{2}\left\|(-\Delta)^{\frac{s}{2}}\phi\right\|_2^2-\frac{1}{p}\|\phi\|_p^p\label{J}
\end{align}
and
\begin{align}
I(\phi)=\langle J'(\phi),\phi\rangle=\left\|(-\Delta)^{\frac{s}{2}}\phi\right\|_2^2-\|\phi\|_p^p.\label{I}
\end{align}
Here, $J'(\phi)$ denotes the Fr'{e}chet derivative of $J(\phi)$, $\langle\cdot,\cdot\rangle$ represents the dual product between $H^{-s}(\Omega)$ and $H_0^s(\Omega)$, and $H^{-s}(\Omega)$ is the dual space of $H_0^s(\Omega)$. Then \eqref{C} shows that $J$ and $I$ given above are well-defined in $H_0^s(\Omega)$. There holds%It is easy to check that
\begin{equation}\label{J=I}
\begin{split}
J(\phi)&=\frac{1}{p}I(\phi)+\frac{p-2}{2p}\left\|(-\Delta)^{\frac{s}{2}}\phi\right\|_2^2\\
&=\frac{1}{2}I(\phi)+\frac{p-2}{2p}\|\phi\|_p^p.
\end{split}
\end{equation}
Moreover, define the Nehari manifold as
\begin{equation}\label{N}
\mathcal{N}=\{\phi\in H_0^s(\Omega)\setminus\{0\}\mid I(\phi)=0\}
\end{equation}
and the potential well depth as
\begin{equation}\label{d}
d=\inf_{\phi\in\mathcal{N}} J(\phi).
\end{equation}
We point out that $d$ is positive (see Lemma \ref{d1}).

To study \eqref{1.1}, we recall the definition of weak solutions for \eqref{1.1}:
\begin{defn}\label{weaksolution} \cite{MR4180066}
Assume that \eqref{p} holds. Define a function $u\in L^\infty\left(0,T;H_0^s(\Omega)\right)$ with $u_t\in L^2(0,T;H_0^s(\Omega))$ as a weak solution of problem \eqref{1.1} on $\Omega\times[0,T)$, if
\begin{equation}\label{weak}
\left(|x|^{-2s}u_t,\nu\right)+\left((-\Delta)^{\frac{s}{2}} u,(-\Delta)^{\frac{s}{2}} \nu\right)+\left((-\Delta)^{\frac{s}{2}} u_t,(-\Delta)^{\frac{s}{2}} \nu\right)=\left(|u|^{p-2}u,\nu\right)
\end{equation}
and $u(x,0)=u_0(x)\in H_0^s(\Omega)$ for any $\nu\in H_0^s(\Omega)$ and a.e. $t\in(0,T)$.
\end{defn}

We recall Theorem 1.1 in \cite{MR4180066}, which provides the local existence of weak solutions for \eqref{1.1}.
\begin{thm}\label{local}
Let \eqref{p} hold and $u_0\in H_0^s(\Omega)$. Then \eqref{1.1} admits a unique local weak solution $u\in L^\infty\left(0,T;H_0^s(\Omega)\right)$ with $u_t\in L^2(0,T;H_0^s(\Omega))$.
\end{thm}
\begin{remark}
Since $u\in L^\infty\left(0,T;H_0^s(\Omega)\right)\subset L^2(0,T;H_0^s(\Omega))$ and $u_t\in L^2(0,T;H_0^s(\Omega))$, we obtain $u\in H^1\left(0,T;H_0^s(\Omega)\right)$, which together with $H^1(0,T)\hookrightarrow C(0,T)$, implies $u\in C\left(0,T;H_0^s(\Omega)\right)$.
\end{remark}

Next we provide some useful lemmas. In the sequel of the present paper, let $u=u(t)$, $t\in[0,T)$ be the weak solution of \eqref{1.1}, where $T$ is the maximal existence time of $u$. The first lemma shows that $d$ defined in \eqref{d} is a positive constant.
\begin{lemma}\label{d1}
Assume that \eqref{p} holds. Then
\begin{equation}\label{djingquezhi}
d=\frac{p-2}{2p}C^{\frac{2p}{2-p}},
\end{equation}
where $C$ is a constant given in \eqref{C}.
\end{lemma}

\begin{proof}
This proof is similar to the one of \cite[Lemma 2.2]{MR4104462} and so we omit it here.
\end{proof}

The next lemma gives the equality related to $I$ and the energy equality.
\begin{lemma}
Let \eqref{p} hold and $u_0\in H_0^s(\Omega)$. Then, for $t\in[0,T)$, it holds that
\begin{equation}\label{dtI}
\frac{d}{dt}\left(\int_\Omega\frac{u^2}{|x|^{2s}}dx+\left\|(-\Delta)^{\frac{s}{2}}u\right\|_2^2\right)=-2I(u)
\end{equation}
and
\begin{equation}\label{Ju0}
\int_0^t\int_\Omega\frac{u_\tau^2}{|x|^{2s}}dxd\tau+\int_0^t\left\|(-\Delta)^{\frac{s}{2}}u_\tau\right\|_2^2d\tau+J(u)=J(u_0)
\end{equation}
\end{lemma}

\begin{proof}
Taking $\nu=u$ in \eqref{weak} and making use of \eqref{I} indicate \eqref{dtI}. Taking $\nu=u_t$ in \eqref{weak} and integrating from $0$ to $t$, then by \eqref{J}, equality \eqref{Ju0} holds easily.
\end{proof}

The following lemma plays a key role in the study of the global existence of solutions.
\begin{lemma}\label{W}
Let \eqref{p} hold and $u_0\in H_0^s(\Omega)$. If $J(u_0)<d$ and $I(u_0)>0$, then it holds that
\begin{equation*}
u\in \mathcal{W}:=\{\phi\in H_0^s(\Omega)\mid I(\phi)>0\}\cup\{0\}, \quad  \forall t\in[0,T).
\end{equation*}
\end{lemma}
\begin{proof}
We only need to verify $u\in \mathcal{W}$ for all $t\in(0,T)$, as it is obvious that $u_0\in \mathcal{W}$ by $I(u_0)>0$.
Arguing by contradiction, if not, then there is a point $t_0\in(0,T)$ satisfying $u(t_0)\in H_0^s(\Omega)\setminus\{0\}$ and $I(u(t_0))=0$, which means $u(t_0)\in\mathcal{N}$. Then by \eqref{d}, we arrive at
\begin{equation}\label{111}
J(u(t_0))\geq d,
\end{equation}
From \eqref{Ju0}, one obtains
\begin{equation*}
J(u(t_0))\leq J(u_0)<d
\end{equation*}
which is contradictory to \eqref{111}.
\end{proof}

We will finish this section with a significant lemma, which is used to study the blow-up properties in the case of $0\leq J(u_0)<d$.
\begin{lemma}\label{le2}
Let \eqref{p} hold and $u_0\in H_0^s(\Omega)$. Set contants
\begin{equation*}
\alpha_1:=C^{\frac{2}{2-p}},\quad
\eta:=\left(\frac{p}{2}-p\alpha_1^{-p}J(u_0)\right)^{\frac{1}{p-2}},
\end{equation*}
where $C$ is a constant given in \eqref{C}. Assume that $I(u_0)<0$ and $0\leq J(u_0)<d$ are
satisfied. Then there exists a constant $\alpha_2>\alpha_1$ satisfying
\begin{equation}\label{9}
\|u\|_{p}\geq\alpha_2,\quad \left\|(-\Delta)^{\frac{s}{2}}u\right\|_2\geq\frac{\alpha_2}{C}, \quad \forall t\in[0,T)
\end{equation}
and
\begin{equation}\label{theta0}
\frac{\alpha_2}{\alpha_1}\geq\eta>1.
\end{equation}
\end{lemma}

\begin{proof}
Since \eqref{C} and $I(u_0)<0$, we have
\begin{equation*}
\|u_0\|_p^p>\left\|(-\Delta)^{\frac{s}{2}}u_0\right\|_2^2\geq C^{-2}\|u_0\|_p^2,
\end{equation*}
which implies $\|u_0\|_p>\alpha_1$. By \eqref{C} and \eqref{J}, we obtain
\begin{equation}\label{le21}
J(u)\geq \frac{1}{2C^2}\|u\|_p^2-\frac{1}{p}\|u\|_p^p.
\end{equation}
Setting
\begin{equation}\label{V}
V(\alpha)=\frac{1}{2C^2}\alpha^2-\frac{1}{p}\alpha^p
\end{equation}
for $\alpha\geq0$, one has that $V(\alpha)$ increases in $(0,\alpha_1)$ and decreases in $(\alpha_1,\infty)$ where \begin{equation*}
V(\alpha_1)=\frac{p-2}{2p}C^{\frac{2p}{2-p}}=d.
\end{equation*}

It follows from $J(u_0)<d$, \eqref{le21} and \eqref{V} that there is a constant $\alpha_2>\alpha_1$ satisfying
\begin{equation}\label{222}
V(\alpha_2)=J(u_0)\geq V(\|u_0\|_p),
\end{equation}
which, together with $\|u_0\|_p>\alpha_1$, yields
\begin{equation*}
\|u_0\|_p\geq\alpha_2.
\end{equation*}
We are now going to show $\|u\|_p\geq\alpha_2$ for $t\in(0,T)$. If not, then by $\alpha_1<\alpha_2$, we may take $t_0\in(0,T)$ satisfying $\alpha_1<\|u(t_0)\|_p<\alpha_2$. By \eqref{le21} and \eqref{222}, we derive
\begin{equation*}
J(u_0)=V(\alpha_2)<V(\|u(t_0)\|_p)\leq J(u(t_0))
\end{equation*}
and from \eqref{Ju0} that
\begin{equation*}
J(u_0)\geq J(u(t_0)),
\end{equation*}
the two are contradictory. We further obtain from \eqref{C} that
\begin{equation*}
C\left\|(-\Delta)^{\frac{s}{2}}u\right\|_2\geq\|u\|_p\geq\alpha_2.
\end{equation*}
Consequently \eqref{9} holds.

Next, we prove \eqref{theta0}. Let
\begin{equation*}
\alpha_0:=\frac{\alpha_2}{\alpha_1}>1.
\end{equation*}
Then by applying \eqref{222} and $\alpha_1=C^{\frac{2}{2-p}}$, we derive that
\begin{equation*}
\begin{split}
J(u_0)&=V(\alpha_1\alpha_0)\\
&=\alpha_1^2\alpha_0^2\left(\frac{1}{2C^2}-\frac{\alpha_1^{p-2}\alpha_0^{p-2}}{p}\right)\\
&=\alpha_0^2C^{-\frac{2p}{p-2}}\left(\frac{1}{2}-\frac{\alpha_0^{p-2}}{p}\right).
\end{split}\end{equation*}
One has by $J(u_0)\geq0$ and $\alpha_0>1$ that
\begin{equation*}
\frac{1}{2}-\frac{\alpha_0^{p-2}}{p}=\frac{C^{\frac{2p}{p-2}}J(u_0)}{\alpha_0^2}\leq C^\frac{2p}{p-2}J(u_0).
\end{equation*}
Then we obtain
\begin{equation*}
\alpha_0\geq\left(\frac{p}{2}-pC^{\frac{2p}{p-2}}J(u_0)\right)^{\frac{1}{p-2}}>\left(\frac{p}{2}-\frac{p-2}{2}\right)^{\frac{1}{p-2}}=1,
\end{equation*}
where the last inequality above follows from $J(u_0)<d$ and \eqref{djingquezhi}.
\end{proof}

\section{Main results}

\quad\quad Our main results are focused on in this section.
The first theorem is about the existence of global solutions to \eqref{1.1} as well as the exponential decay
estimates of both global solutions and energy functional at low initial energy.

\begin{thm}\label{1111}
Let \eqref{p} hold and $u_0\in H_0^s(\Omega)$. If $J(u_0)<d$ and $I(u_0)>0$ are satisfied, then the unique local weak solution $u$ of \eqref{1.1} is global. %there exists a global weak solution for problem \eqref{1.1}.
In addition, for all $t\in[0,\infty)$, there hold
\begin{align*}
&\|u\|_2\leq C\left(\int_\Omega\frac{u_0^2}{|x|^{2s}}dx+\left\|(-\Delta)^{\frac{s}{2}}u_0\right\|_2^2\right)^{\frac{1}{2}}e^{-\frac{\delta}{2}t},\\
&\|u\|_p\leq C\left(\int_\Omega\frac{u_0^2}{|x|^{2s}}dx+\left\|(-\Delta)^{\frac{s}{2}}u_0\right\|_2^2\right)^{\frac{1}{2}}e^{-\frac{\delta}{2}t},\\
&\int_\Omega\frac{u^2}{|x|^{2s}}dx+\left\|(-\Delta)^{\frac{s}{2}}u\right\|_2^2\leq \left(\int_\Omega\frac{u_0^2}{|x|^{2s}}dx+\left\|(-\Delta)^{\frac{s}{2}}u_0\right\|_2^2\right)^{\frac{1}{2}}e^{-\frac{\delta}{2}t},\\
&J(u)\leq\left(J(u_0)+\int_\Omega\frac{u_0^2}{|x|^{2s}}dx+\left\|(-\Delta)^{\frac{s}{2}}u_0\right\|_2^2\right)e^{-\left(\frac{3p-2}{p-2}+C_*\right)\kappa t},
\end{align*}
where $C$ is defined in \eqref{C}, and
\begin{align*}
&\kappa=\frac{4p\delta}{2\delta+p-2},\\
&\delta=\frac{2}{1+C_*}\left[1-\left(\frac{J(u_0)}{d}\right)^{\frac{p-2}{2}}\right].
\end{align*}
Here, $C_*$ is defined in \eqref{f111-22}.
\end{thm}
\begin{proof}
Due to $I(u_0)>0$ and Lemma \ref{W} indicating $I(u)\geq0$ for all $t\in[0,T)$, from \eqref{J=I} and \eqref{Ju0}, one has
\begin{equation*}
\begin{split}
d&>J(u_0)\\
&=\int_0^t\int_\Omega\frac{u_\tau^2}{|x|^{2s}}dxd\tau+\int_0^t\left\|(-\Delta)^{\frac{s}{2}}u_\tau\right\|_2^2d\tau+\frac{p-2}{2p}\left\|(-\Delta)^{\frac{s}{2}}u\right\|_2^2+\frac{1}{p}I(u)\\
&\geq\int_0^t\int_\Omega\frac{u_\tau^2}{|x|^{2s}}dxd\tau+\int_0^t\left\|(-\Delta)^{\frac{s}{2}}u_\tau\right\|_2^2d\tau+\frac{p-2}{2p}\left\|(-\Delta)^{\frac{s}{2}}u\right\|_2^2,\quad t\in[0,T),
\end{split}
\end{equation*}
which combined with Theorem \ref{local}, illustrates that $u$ exists globally.
%illustrates $\int_0^t\int_\Omega\frac{u_\tau^2}{|x|^{2s}}dxd\tau+\int_0^t\left\|(-\Delta)^{\frac{s}{2}}u_\tau\right\|_2^2d\tau+\left\|(-\Delta)^{\frac{s}{2}}u\right\|_2^2$ is uniformly bounded in time. Consequently $T=\infty$, that is $u(t)$ exists globally.
Therefore, $I(u)\geq0$ holds for all $t\in[0,\infty)$. From \eqref{Ju0} and \eqref{J=I}, we obtain
\begin{equation}\label{fffff1-22}
\begin{split}
J(u_0)\geq J(u)&=\frac{1}{p}I(u)+\frac{p-2}{2p}\left\|(-\Delta)^{\frac{s}{2}}u\right\|_2^2\\
&\geq\frac{p-2}{2p}\left\|(-\Delta)^{\frac{s}{2}}u\right\|_2^2,
\end{split}
\end{equation}
By \eqref{C}, it follows that
\begin{equation*}
\|u\|_p\leq C\left\|(-\Delta)^{\frac{s}{2}}u\right\|_2\leq C\left(\frac{2p}{p-2}J(u_0)\right)^{\frac{1}{2}},
\end{equation*}
Combining with \eqref{djingquezhi}, we deduce
\begin{equation}\label{ffff1-22}
\begin{split}
\|u\|_p^p&=\|u\|_p^{p-2}\|u\|_p^2\\
&\leq C^{p-2}\left(\frac{2p}{p-2}J(u_0)\right)^{\frac{p-2}{2}}C^2\left\|(-\Delta)^{\frac{s}{2}}u\right\|_2^2\\
&=\left(C^{\frac{2p}{p-2}}\frac{2p}{p-2}J(u_0)\right)^{\frac{p-2}{2}}\left\|(-\Delta)^{\frac{s}{2}}u\right\|_2^2\\
&=\left(\frac{J(u_0)}{d}\right)^{\frac{p-2}{2}}\left\|(-\Delta)^{\frac{s}{2}}u\right\|_2^2.
\end{split}
\end{equation}
Using \eqref{dtI}, \eqref{I}, \eqref{f111-22}, and \eqref{ffff1-22}, we derive
\begin{equation}\label{fffff2-22}
\begin{split}
&\quad\frac{d}{dt}\left(\int_\Omega\frac{u^2}{|x|^{2s}}dx+\left\|(-\Delta)^{\frac{s}{2}}u\right\|_2^2\right)\\
&=-2I(u)\\
&=-2\left(\left\|(-\Delta)^{\frac{s}{2}}u\right\|_2^2-\|u\|_p^p\right)\\
&\leq-2\left(1-\left(\frac{J(u_0)}{d}\right)^{\frac{p-2}{2}}\right)\left\|(-\Delta)^{\frac{s}{2}}u\right\|_2^2\\
&\leq-\frac{2}{1+C_*}\left[1-\left(\frac{J(u_0)}{d}\right)^{\frac{p-2}{2}}\right]\left(\int_\Omega\frac{u^2}{|x|^{2s}}dx+\left\|(-\Delta)^{\frac{s}{2}}u\right\|_2^2\right)\\
&=:-\delta\left(\int_\Omega\frac{u^2}{|x|^{2s}}dx+\left\|(-\Delta)^{\frac{s}{2}}u\right\|_2^2\right),
\end{split}
\end{equation}
Hence, we obtain
\begin{equation*}
\int_\Omega\frac{u^2}{|x|^{2s}}dx+\left\|(-\Delta)^{\frac{s}{2}}u\right\|_2^2\leq\left(\int_\Omega\frac{u_0^2}{|x|^{2s}}dx+\left\|(-\Delta)^{\frac{s}{2}}u_0\right\|_2^2\right)e^{-\delta t},
\end{equation*}
%which implies
%\begin{equation*}
%\|u\|_{H_0^s}=\left\|(-\Delta)^{\frac{s}{2}}u\right\|_2\leq\left(\int_\Omega\frac{u_0^2}{|x|^{2s}}dx+\left\|(-\Delta)^{\frac{s}{2}}u_0\right\|_2^2\right)^{\frac{1}{2}}e^{-\frac{\delta}{2}t},
%\end{equation*}
Combining with \eqref{C}, we further have
\begin{align*}
&\|u\|_2\leq C\left(\int_\Omega\frac{u_0^2}{|x|^{2s}}dx+\left\|(-\Delta)^{\frac{s}{2}}u_0\right\|_2^2\right)^{\frac{1}{2}}e^{-\frac{\delta}{2}t},\\
&\|u\|_p\leq C\left(\int_\Omega\frac{u_0^2}{|x|^{2s}}dx+\left\|(-\Delta)^{\frac{s}{2}}u_0\right\|_2^2\right)^{\frac{1}{2}}e^{-\frac{\delta}{2}t}.
\end{align*}

Define
\begin{equation*}
\mathcal{L}(t):=J(u)+\int_\Omega\frac{u^2}{|x|^{2s}}dx+\left\|(-\Delta)^{\frac{s}{2}}u\right\|_2^2,\quad t\geq0.
\end{equation*}
It follows from \eqref{f111-22} and \eqref{fffff1-22} that
\begin{equation}\label{ffff2-22}
\begin{split}
\mathcal{L}(t)&\leq J(u)+(1+C_*)\left\|(-\Delta)^{\frac{s}{2}}u\right\|_2^2\\
&\leq\left(\frac{3p-2}{p-2}+C_*\right)J(u).
\end{split}\end{equation}
From \eqref{Ju0}, \eqref{dtI}, \eqref{J=I} and \eqref{fffff2-22}, we conclude that for any constant $\kappa>0$,
\begin{equation}\label{ffff3-22}
\begin{split}
\mathcal{L}'(t)&=\frac{d}{dt}J(u)+\frac{d}{dt}\left(\int_\Omega\frac{u^2}{|x|^{2s}}dx+\left\|(-\Delta)^{\frac{s}{2}}u\right\|_2^2\right)\\
&=-\left(\int_\Omega\frac{u_t^2}{|x|^{2s}}dx+\left\|(-\Delta)^{\frac{s}{2}}u_t\right\|_2^2\right)-2I(u)\\
&\leq-2I(u)-\kappa J(u)+\kappa\left(\frac{1}{p}I(u)+\frac{p-2}{2p}\left\|(-\Delta)^{\frac{s}{2}}u\right\|_2^2\right)\\
&\leq\left(\frac{\kappa}{p}+\frac{\kappa(p-2)}{p\delta}-2\right)I(u)-\kappa J(u),
\end{split}
\end{equation}
Let
\begin{equation*}
\kappa=\frac{2p\delta}{\delta+p-2}>0,
\end{equation*}
where $\delta$ is defined in \eqref{fffff2-22}. Then from \eqref{ffff3-22} and \eqref{ffff2-22}, we obtain
\begin{equation*}
\mathcal{L}'(t)\leq-\kappa J(u)
\leq-\left(\frac{3p-2}{p-2}+C_*\right)\kappa\mathcal{L}(t).
\end{equation*}
Consequently,
\begin{equation*}
\mathcal{L}(t)\leq \mathcal{L}(0)e^{-\left(\frac{3p-2}{p-2}+C_*\right)\kappa t},
\end{equation*}
which, together with the definition of $\mathcal{L}$, yields
\begin{equation*}
J(u)\leq\left(J(u_0)+\int_\Omega\frac{u_0^2}{|x|^{2s}}dx+\left\|(-\Delta)^{\frac{s}{2}}u_0\right\|_2^2\right)e^{-\left(\frac{3p-2}{p-2}+C_*\right)\kappa t}.
\end{equation*}
\end{proof}

\begin{remark}
By $J(u_0)<d$, $I(u_0)>0$ and \eqref{J=I}, we obtain $0<J(u_0)<d$. Thus, the decay estimate in Theorem \ref{1111} make sense.
\end{remark}

The second theorem is about the finite time blow-up of solutions and the upper bounds of both blow-up time and rate at low initial energy.

\begin{thm}\label{baopo}
Let \eqref{p} hold and $u_0\in H_0^s(\Omega)$. If $J(u_0)<d$ and $I(u_0)<0$ are satisfied, then the global solution $u$ blows up at a finite time in the sense of
\begin{equation}\label{bp}
\lim_{t\rightarrow T^-}\left(\int_\Omega\frac{u^2}{|x|^{2s}}dx+\left\|(-\Delta)^{\frac{s}{2}}u\right\|_2^2\right)=\infty.
\end{equation}
Additionally, an upper bound of blow-up time satisfies
  \begin{equation*}
  T\leq\frac{\int_\Omega\frac{u_0^2}{|x|^{2s}}dx+\left\|(-\Delta)^{\frac{s}{2}}u_0\right\|_2^2}{K(K-2)\mathcal{J}(0)}
  \end{equation*}
  and an upper bound of blow-up rate satisfies
  \begin{equation*}
  \int_\Omega\frac{u^2}{|x|^{2s}}dx+\left\|(-\Delta)^{\frac{s}{2}}u\right\|_2^2\leq
               \left[\frac{\left(\int_\Omega\frac{u_0^2}{|x|^{2s}}dx
               +\left\|(-\Delta)^{\frac{s}{2}}u_0\right\|_2^2\right)^{\frac{K}{2}}}{K(K-2)\mathcal{J}(0)}\right]^{\frac{2}{K-2}}(T-t)^{-\frac{2}{K-2}},
  \end{equation*}
  where
   \begin{equation*}
  \mathcal{J}(0)=\left\{
               \begin{array}{ll}
               \ds  -J(u_0),&\hbox{~when~}J(u_0)<0; \\
               \vspace{-0.1in}\\
               \ds  d-J(u_0),&\hbox{~when~}0\leq J(u_0)<d,
               \end{array}
             \right.
  \end{equation*}
and
   \begin{equation}\label{CC1}
  K=\left\{
               \begin{array}{ll}
               \ds  p,&\hbox{~when~}J(u_0)<0; \\
               \vspace{-0.1in}\\
               \ds  \frac{\left(\eta^{p}-1\right)(p-2)}{\eta^{p}}+2,&\hbox{~when~}0\leq J(u_0)<d.
               \end{array}
             \right.
  \end{equation}
  Here, $\eta$ is a constant defined in Lemma \ref{le2}.
\end{thm}

\begin{proof}
For any $t\in[0,T)$, set functions
\begin{equation}\label{H}
\mathcal{S}(t)=\frac{1}{2}\left(\int_\Omega\frac{u^2}{|x|^{2s}}dx+\left\|(-\Delta)^{\frac{s}{2}}u\right\|_2^2\right)
\end{equation}
and
\begin{equation}\label{GG}
  \mathcal{J}(t)=\left\{
               \begin{array}{ll}
               \ds  -J(u),&\hbox{~when~}J(u_0)<0; \\
               \vspace{-0.1in}\\
               \ds  d-J(u),&\hbox{~when~}0\leq J(u_0)<d.
               \end{array}
             \right.
  \end{equation}
When $J(u_0)<0$ is satisfied, taking into account \eqref{Ju0} and \eqref{J}, we have
\begin{equation}\label{th11}
0<\mathcal{J}(0)\leq \mathcal{J}(t)=-\frac{1}{2}\left\|(-\Delta)^{\frac{s}{2}}u\right\|_2^2+\frac{1}{p}\|u\|_{p}^{p}\leq\frac{1}{p}\|u\|_{p}^{p}.
\end{equation}
Using \eqref{H}, \eqref{dtI}, \eqref{J=I}, \eqref{GG} and \eqref{th11}, one obtains
\begin{equation}\label{th12}
\mathcal{S}'(t)=-2J(u)+\frac{p-2}{p}\|u\|_{p}^{p}\geq p\mathcal{J}(t).
\end{equation}
When $0\leq J(u_0)<d$ is satisfied, we can deduce from  \eqref{GG}, \eqref{J}, \eqref{djingquezhi} and \eqref{C}, that
\begin{equation}\label{th0000}
\begin{split}
0<\mathcal{J}(0)
&\leq \mathcal{J}(t)\\
&=d-J(u)\\
&=d-\frac{1}{2}\left\|(-\Delta)^{\frac{s}{2}} u\right\|_2^2+\frac{1}{p}\|u\|_{p}^{p}\\
&\leq\frac{p-2}{2p}C^{-\frac{2p}{p-2}}-\frac{1}{2C^2}\|u\|_{p}^2+\frac{1}{p}\|u\|_{p}^{p}.
\end{split}\end{equation}
According to Lemma \ref{le2}, it holds that
\begin{equation*}
-\frac{1}{2C^2}\|u\|_{p}^2\leq-\frac{\alpha_1^2}{2C^2}=-\frac{1}{2}C^{-\frac{2p}{p-2}},
\end{equation*}
which together with \eqref{th0000}, indicates
\begin{equation}\label{th13}
\mathcal{J}(t)\leq\left(\frac{p-2}{2p}-\frac{1}{2}\right)C^{-\frac{2p}{p-2}}+\frac{1}{p}\|u\|_{p}^{p}
=-\frac{1}{p}C^{-\frac{2p}{p-2}}+\frac{1}{p}\|u\|_{p}^{p}\leq\frac{1}{p}\|u\|_{p}^{p}.
\end{equation}
It follows from \eqref{H}, \eqref{dtI}, \eqref{J=I}, \eqref{GG}, \eqref{djingquezhi}, Lemma \ref{le2} and \eqref{th13} that
\begin{equation}\label{th14}
\begin{split}
\mathcal{S}'(t)&=-2J(u)+\frac{p-2}{p}\|u\|_{p}^{p}\\
&=2\mathcal{J}(t)-\frac{p-2}{p}\left(C^{-\frac{2p}{p-2}}-\|u\|_{p}^{p}\right)\\
&=2\mathcal{J}(t)-\frac{p-2}{p}\left(\left(\frac{\alpha_1}{\alpha_2}\right)^{p}\alpha_2^{p}-\|u\|_{p}^{p}\right)\\
&\geq 2\mathcal{J}(t)-\frac{p-2}{p}\left(\frac{\|u\|_{p}^{p}}{\eta^{p}}-\|u\|_{p}^{p}\right)\\
&\geq \left[\frac{\left(\eta^{p}-1\right)(p-2)}{\eta^{p}}+2\right]\mathcal{J}(t).
\end{split}\end{equation}
Consequently, according to \eqref{th12} and \eqref{th14}, we have
\begin{equation}\label{th15}
\mathcal{S}'(t)\geq K\mathcal{J}(t)>0,
\end{equation}
where $K$ is seen as \eqref{CC1}.

Since \eqref{Ju0} guarantees
\begin{equation*}
\mathcal{J}'(t)=\int_\Omega\frac{u_t^2}{|x|^{2s}}dx+\left\|(-\Delta)^{\frac{s}{2}} u_t\right\|_2^2,
\end{equation*}
by \eqref{H}, \eqref{th15} and Cauchy-Schwarz's inequality, we obtain
\begin{equation*}\begin{split}
\mathcal{S}(t)\mathcal{J}'(t)&=\frac{1}{2}\left(\int_\Omega\frac{u^2}{|x|^{2s}}dx+\left\|(-\Delta)^{\frac{s}{2}} u\right\|_2^2\right)\left(\int_\Omega\frac{u_t^2}{|x|^{2s}}dx+\left\|(-\Delta)^{\frac{s}{2}} u_t\right\|_2^2\right)\\
&\geq\frac{1}{2}\left(\int_\Omega\frac{uu_t}{|x|^{2s}}dx\right)^2+\int_\Omega\frac{uu_t}{|x|^{2s}}dx\int_\Omega(-\Delta)^{\frac{s}{2}} u(-\Delta)^{\frac{s}{2}} u_tdx\\
&\quad+\frac{1}{2}\left(\int_\Omega(-\Delta)^{\frac{s}{2}} u(-\Delta)^{\frac{s}{2}} u_tdx\right)^2\\
&=\frac{1}{2}(\mathcal{S}'(t))^2\\
&\geq\frac{K}{2}\mathcal{S}'(t)\mathcal{J}(t),
\end{split}\end{equation*}
which implies
\begin{equation*}
\frac{\mathcal{J}'(t)}{\mathcal{J}(t)}\geq\frac{K\mathcal{S}'(t)}{2\mathcal{S}(t)}.
\end{equation*}
Integrating it from $0$ to $t$ and using \eqref{th15} yield
\begin{equation}\label{th18}
\frac{\mathcal{S}'(t)}{(\mathcal{S}(t))^{\frac{K}{2}}}\geq\frac{K\mathcal{J}(0)}{(\mathcal{S}(0))^{\frac{K}{2}}}.
\end{equation}
Continuing to integrate \eqref{th18} over $[0,t]$, we reach
\begin{equation}\label{th19}
(\mathcal{S}(t))^{-\frac{K-2}{2}}\leq(\mathcal{S}(0))^{-\frac{K-2}{2}}-\frac{K\mathcal{J}(0)(K-2)}{2(\mathcal{S}(0))^{\frac{K}{2}}}t.
\end{equation}
Taking into account $K>2$, one can find that \eqref{th19} cannot hold when $t$ is large enough. Therefore, the solution $u$ blows up at a finite time with
\begin{equation}\label{Swuqiong}
\lim\limits_{t\rightarrow T^-}\mathcal{S}(t)=\infty,
\end{equation}
and $T$ has the upper bound
\begin{equation*}\begin{split}
T&\leq\frac{2\mathcal{S}(0)}{K(K-2)\mathcal{J}(0)}\\
&=\frac{\int_\Omega\frac{u_0^2}{|x|^{2s}}dx+\left\|(-\Delta)^{\frac{s}{2}}u_0\right\|_2^2}{K(K-2)\mathcal{J}(0)}.
\end{split}\end{equation*}

In addition, integrating \eqref{th18} again from $t$ to $T$ and making use of \eqref{Swuqiong} give
\begin{equation*}
S(t)\leq\left[\frac{2(S(0))^{\frac{K}{2}}}{K(K-2)\mathcal{J}(0)}\right]^{\frac{2}{K-2}}(T-t)^{-\frac{2}{K-2}},
\end{equation*}
which together with \eqref{H}, implies the upper bound of blow-up rate
\begin{equation*}
  \int_\Omega\frac{u^2}{|x|^{2s}}dx+\left\|(-\Delta)^{\frac{s}{2}}u\right\|_2^2\leq
               \left[\frac{\left(\int_\Omega\frac{u_0^2}{|x|^{2s}}dx
  +\left\|(-\Delta)^{\frac{s}{2}}u_0\right\|_2^2\right)^{\frac{K}{2}}}{K(K-2)\mathcal{J}(0)}\right]^{\frac{2}{K-2}}
  (T-t)^{-\frac{2}{K-2}},
  \end{equation*}
which completes the proof.
\end{proof}

Next, it remains to estimate the lower bounds of blow-up time and rate, as shown below, which are important components of blow-up properties.
\begin{thm}\label{xiajie}
Let \eqref{p} hold and $u_0\in H_0^s(\Omega)$. Assume that the solution $u$ blows up at a finite time $T$ with \eqref{bp}. Then a lower bound of $T$ can be given by
\begin{equation*}
T\geq\frac{1}{C^{p}(p-2)}\left(\int_\Omega\frac{u_0^2}{|x|^{2s}}dx+\left\|(-\Delta)^{\frac{s}{2}} u_0\right\|_2^2\right)^{\frac{2-p}{2}},
\end{equation*}
and a lower bound of blow-up rate can be given by
\begin{equation*}
\int_\Omega\frac{u^2}{|x|^{2s}}dx+\left\|(-\Delta)^{\frac{s}{2}} u\right\|_2^2\geq C^{\frac{2p}{2-p}}(p-2)^{\frac{2}{2-p}}(T-t)^{\frac{2}{2-p}},
\end{equation*}
where $C>0$ is the constant given in \eqref{C}.
\end{thm}

\begin{proof}
For the function
\begin{equation*}
\mathcal{S}(t)=\frac{1}{2}\left(\int_\Omega\frac{u^2}{|x|^{2s}}dx+\left\|(-\Delta)^{\frac{s}{2}}u\right\|_2^2\right),
\end{equation*}
we define in \eqref{H}, one obtains
\begin{equation*}
\lim\limits_{t\rightarrow T^-}\mathcal{S}(t)=\infty
\end{equation*}
from the proof of Theorem \ref{baopo}.
It follows from \eqref{I}, \eqref{dtI} and \eqref{C} that
\begin{equation*}
\begin{split}
\mathcal{S}'(t)&=-\left\|(-\Delta)^{\frac{s}{2}} u\right\|_2^2+\|u\|_p^p\\
&\leq C^p\left\|(-\Delta)^{\frac{s}{2}} u\right\|_2^p\leq 2^{\frac{p}{2}}C^p(\mathcal{S}(t))^{\frac{p}{2}}.
\end{split}\end{equation*}
Obviously,
\begin{equation}\label{th21}
\frac{\mathcal{S}'(t)}{(\mathcal{S}(t))^{\frac{p}{2}}}\leq2^{\frac{p}{2}}C^p.
\end{equation}
Integrating from $0$ to $t$ turns out
\begin{equation*}
-\frac{2}{p-2}(\mathcal{S}(t))^{\frac{2-p}{2}}\leq2^{\frac{p}{2}}C^pt-\frac{2}{p-2}(\mathcal{S}(0))^{\frac{2-p}{2}}.
\end{equation*}
Then taking $t\rightarrow T^-$ and applying $\lim\limits_{t\rightarrow T^-}\mathcal{S}(t)=\infty$ yield
\begin{equation*}
\begin{split}
T&\geq \frac{2^{\frac{2-p}{2}}}{C^p(p-2)}(\mathcal{S}(0))^{\frac{2-p}{2}}\\
&=\frac{1}{C^{p}(p-2)}\left(\int_\Omega\frac{u_0^2}{|x|^{2s}}dx+\left\|(-\Delta)^{\frac{s}{2}} u_0\right\|_2^2\right)^{\frac{2-p}{2}}.
\end{split}\end{equation*}
We integrate \eqref{th21} again from $t$ to $T$ and it holds that
\begin{equation*}
\mathcal{S}(t)\geq\frac{1}{2}C^{\frac{2p}{2-p}}(p-2)^{\frac{2}{2-p}}(T-t)^{\frac{2}{2-p}}.
\end{equation*}
Reviewing the definition of $\mathcal{S}(t)$, we conclude that
\begin{equation*}
\int_\Omega\frac{u^2}{|x|^{2s}}dx+\left\|(-\Delta)^{\frac{s}{2}} u\right\|_2^2\geq C^{\frac{2p}{2-p}}(p-2)^{\frac{2}{2-p}}(T-t)^{\frac{2}{2-p}},
\end{equation*}
which completes the proof.
\end{proof}

%The following first lemma will provide a positive answer.
%According to \eqref{d}, a nature question is whether $d$ is achievable. Additionally, from Theorem \ref{} we notice that the global solution decays to $0$ at the exponential rate as $t\rightarrow\infty$ when satisfying some initial conditions, how about the general global solutions? For this purpose,

Finally, we introduce the stationary problem corresponding to \eqref{1.1}, i.e., the following boundary value problem
\begin{align}\label{2.1}
 \left\{\begin{array}{ll}
     \ds (-\Delta)^su=|u|^{p-2}u,\quad &x\in\Omega,\\
     \ds u(x)=0,\quad &x\in\partial\Omega,
    \end{array}\right.
\end{align}
and establish the existence of ground-state solutions to \eqref{2.1}. Moreover, we prove that the global solution of \eqref{1.1} strongly converge to the solution of \eqref{2.1} as time approaches infinity.
We call $u\in H_0^s(\Omega)$ is a solution of \eqref{2.1} if
\begin{equation}\label{ws}
\langle J'(u),\nu\rangle=\left((-\Delta)^{\frac{s}{2}} u,(-\Delta)^{\frac{s}{2}}\nu\right)-\left(|u|^{p-2}u,\nu\right)=0
\end{equation}
for any $\nu\in H_0^s(\Omega)$. The set $\Phi$ is defined as the collection of all solutions to problem \eqref{2.1}:
\begin{equation*}
\Phi=\{\phi\in H_0^s(\Omega):\langle J'(\phi),\nu\rangle=0,~\forall\nu\in H_0^s(\Omega)\}.
\end{equation*}

\begin{thm}\label{jitaijie}
Assume that \eqref{p} holds. %Let $\mathcal{N}$ be defined in \eqref{N}.
Then there exists a function $\phi_0\in \mathcal{N}$ satisfying
\begin{equation}\label{555}
J(\phi_0)=\inf_{\phi\in\mathcal{N}} J(\phi)=d.
\end{equation}
Moreover, $\phi_0$ is a ground-state solution to \eqref{2.1}, that is
\begin{equation}\label{777}
J(\phi_0)=\inf\limits_{\phi\in\Phi\backslash \{0\}} J(\phi) \hbox{~~~and~~~} \phi_0\in\Phi\backslash \{0\}.
\end{equation}
\end{thm}

\begin{thm}\label{ccc-22}
Assume that \eqref{p} holds. Let $u=u(t)$ be a global solution to \eqref{1.1}. Then there exists $u^*\in \Phi$ and an increasing sequence $\{t_k\}_{k=1}^\infty$ with $t_k\rightarrow\infty$ as $k\rightarrow\infty$ such that
\begin{equation*}
\lim_{k\rightarrow\infty}\|u(t_k)-u^*\|_{H_0^s}=0.
\end{equation*}
\end{thm}

\begin{proof}[Proof of Theorem \ref{jitaijie}]
Firstly, we prove there is a $\phi_0\in\mathcal{N}$ satisfying \eqref{555}. It follows from \eqref{d}, \eqref{J=I} and \eqref{N} that
\begin{equation*}
\begin{split}
d&=\inf_{\phi\in\mathcal{N}}\left(\frac{p-2}{2p}\left\|(-\Delta)^{\frac{s}{2}}\phi\right\|_2^2+\frac{1}{p}I(\phi)\right)\\
&=\frac{p-2}{2p}\inf_{\phi\in\mathcal{N}}\left\|(-\Delta)^{\frac{s}{2}}\phi\right\|_2^2.
\end{split}\end{equation*}
Then there exists a minimizing sequence $\{\phi_k\}_{k=1}^\infty\subset\mathcal{N}$ such that
\begin{equation}\label{5}
\lim\limits_{k\rightarrow\infty} J(\phi_k)=\frac{p-2}{2p}\lim\limits_{k\rightarrow\infty}\left\|(-\Delta)^{\frac{s}{2}}\phi_k\right\|_2^2=d.
\end{equation}
This shows that there is a constant $\Xi$ independent of $k$ such that
\begin{equation*}
\left\|(-\Delta)^{\frac{s}{2}}\phi_k\right\|_2^2\leq\Xi
\end{equation*}
hold for all $k=1,2,\cdots$, which, together with $H_0^s(\Omega)$ is a reflexive Banach space and $H_0^s(\Omega)\hookrightarrow L^p(\Omega)$ compactly, indicates that there is a subsequence of $\{\phi_k\}_{k=1}^\infty$, represented by $\{\phi_k\}_{k=1}^\infty$ again, and an element $\phi_0\in H_0^s(\Omega)$ such that
\begin{equation}\label{w}
\phi_k\rightharpoonup\phi_0 \hbox{ weakly in } H_0^s(\Omega)
\end{equation}
and
\begin{equation}\label{s}
\phi_k\rightarrow\phi_0 \hbox{ strongly in } L^p(\Omega)
\end{equation}
as $k\rightarrow\infty$. By $\{\phi_k\}_{k=1}^\infty\in\mathcal{N}$ and \eqref{N}, we have $I(\phi_k)=0$ for all $k$ and so
\begin{equation}\label{6}
\left\|(-\Delta)^{\frac{s}{2}}\phi_k\right\|_2^2=\|\phi_k\|_p^p.
\end{equation}
It follows from \eqref{w} and \eqref{s} that
\begin{equation}\label{3}
\left\|(-\Delta)^{\frac{s}{2}}\phi_0\right\|_2^2\leq\liminf\limits_{k\rightarrow\infty}\left\|(-\Delta)^{\frac{s}{2}}\phi_k\right\|_2^2
=\lim\limits_{k\rightarrow\infty}\|\phi_k\|_p^p=\|\phi_0\|_p^p.
\end{equation}
and
\begin{equation}\label{66}
\left\|(-\Delta)^{\frac{s}{2}}\phi_0\right\|_2^2
\leq\liminf\limits_{k\rightarrow\infty}\left\|(-\Delta)^{\frac{s}{2}}\phi_k\right\|_2^2
\leq\lim\limits_{k\rightarrow\infty}\left\|(-\Delta)^{\frac{s}{2}}\phi_k\right\|_2^2.
\end{equation}

Now we claim that $I(\phi_0)=0$, i.e., $\left\|(-\Delta)^{\frac{s}{2}}\phi_0\right\|_2^2=\|\phi_0\|_p^p$. If not, then it follows \eqref{3} that $\left\|(-\Delta)^{\frac{s}{2}}\phi_0\right\|_2^2<\|\phi_0\|_p^p$. Clearly, we have $\phi_0\neq0$. Then there is a constant $\lambda>0$ defined by
\begin{equation}\label{lm}
\lambda=\left(\frac{\left\|(-\Delta)^{\frac{s}{2}}\phi_0\right\|_2^2}{\|\phi_0\|_p^p}\right)^{\frac{1}{p-2}}<1
\end{equation}
such that $\lambda\phi_0\in\mathcal{N}$. Hence, from \eqref{d}, we reach
\begin{equation}\label{99}
J(\lambda\phi_0)\geq d.
\end{equation}
It follows from \eqref{J=I}, \eqref{lm}, \eqref{N}, \eqref{66} and \eqref{5} that
\begin{equation*}
\begin{split}
J(\lambda\phi_0)&=\frac{p-2}{2p}\lambda^2\left\|(-\Delta)^{\frac{s}{2}}\phi_0\right\|_2^2+\frac{1}{p}I(\lambda\phi_0)\\
&<\frac{p-2}{2p}\left\|(-\Delta)^{\frac{s}{2}}\phi_0\right\|_2^2\\
&\leq d,
\end{split}\end{equation*}
which conflicts with \eqref{99}. Thus $I(\phi_0)=0$.

By \eqref{s}, \eqref{6} and $I(\phi_0)=0$, we obtain
\begin{equation*}
\lim\limits_{k\rightarrow\infty}\left\|(-\Delta)^{\frac{s}{2}}\phi_k\right\|_2^2=\lim\limits_{k\rightarrow\infty}\|\phi_k\|_p^p
=\|\phi_0\|_p^p=\left\|(-\Delta)^{\frac{s}{2}}\phi_0\right\|_2^2,
\end{equation*}
which, together with \eqref{w}, implies
\begin{equation*}
\phi_k\rightarrow\phi_0 \hbox{ strongly in } H_0^s(\Omega)
\end{equation*}
as $k\rightarrow\infty$. Then we derive from \eqref{J=I}, $I(\phi_0)=0$ and \eqref{5} that
\begin{equation*}
\begin{split}
J(\phi_0)&=\frac{p-2}{2p}\left\|(-\Delta)^{\frac{s}{2}}\phi_0\right\|_2^2+\frac{1}{p}I(\phi_0)\\
&=\frac{p-2}{2p}\lim\limits_{k\rightarrow\infty}\left\|(-\Delta)^{\frac{s}{2}}\phi_k\right\|_2^2\\
&=d,
\end{split}\end{equation*}
which means $\phi_0\neq0$. Combined with $I(\phi_0)=0$, we conclude that $\phi_0\in\mathcal{N}$ and \eqref{555} holds.

Secondly, we prove \eqref{777} is true. It follows from \eqref{555} and $\phi_0\in\mathcal{N}$ that there is a constant $\gamma\in\mathbb{R}$ satisfying
\begin{equation}\label{10}
J'(\phi_0)-\gamma I'(\phi_0)=0
\end{equation}
by using the method of Lagrange multipliers. Then by \eqref{I}, one has
\begin{equation}\label{8}
\gamma\langle I'(\phi_0),\phi_0\rangle=\langle J'(\phi_0),\phi_0\rangle=I(\phi_0)=0,
\end{equation}
and
\begin{equation*}
\begin{split}
\langle I'(\phi_0),\phi\rangle&=\frac{d}{d\sigma}I(\phi_0+\sigma\phi)\bigg|_{\sigma=0}\\
&=\frac{d}{d\sigma}\left[\left\|(-\Delta)^{\frac{s}{2}}(\phi_0+\sigma\phi)\right\|_2^2-\|(\phi_0+\sigma\phi)\|_p^p\right]\bigg|_{\sigma=0}\\
&=\left[2\left((-\Delta)^{\frac{s}{2}}(\phi_0+\sigma\phi),(-\Delta)^{\frac{s}{2}}\phi\right)-p\left(|\phi_0+\sigma\phi|^{p-2}(\phi_0+\sigma\phi),\phi\right)\right]\bigg|_{\sigma=0}\\
&=2\left((-\Delta)^{\frac{s}{2}}\phi_0,(-\Delta)^{\frac{s}{2}}\phi\right)-p\left(|\phi_0|^{p-2}\phi_0,\phi\right).
\end{split}\end{equation*}
Taking $\phi=\phi_0$ above and by $I(\phi_0)=0$, we have
\begin{equation*}
\begin{split}
\langle I'(\phi_0),\phi_0\rangle&=2\left\|(-\Delta)^{\frac{s}{2}}\phi_0\right\|_2^2-p\|\phi_0\|_p^p\\
&=(2-p)\left\|(-\Delta)^{\frac{s}{2}}\phi_0\right\|_2^2\\
&<0.
\end{split}\end{equation*}
By \eqref{8}, we have $\gamma=0$. And $J'(\phi_0)=0$ by \eqref{10}. Then it follows from \eqref{ws} that
\begin{equation*}
\langle J'(\phi_0),\nu\rangle=\left((-\Delta)^{\frac{s}{2}} \phi_0,(-\Delta)^{\frac{s}{2}}\nu\right)-\left(|\phi_0|^{p-2}\phi_0,\nu\right)=0
\end{equation*}
for all $\nu\in H_0^s(\Omega)$, i.e. $\phi_0$ is a solution for \eqref{2.1}.

For any $\phi\in\Phi$, by \eqref{ws}, one has
\begin{equation*}
\left((-\Delta)^{\frac{s}{2}} \phi,(-\Delta)^{\frac{s}{2}}\nu\right)=\left(|\phi|^{p-2}\phi,\nu\right),\quad\forall\nu\in H_0^s(\Omega).
\end{equation*}
%for all $\nu\in H_0^s(\Omega)$.
Taking $\nu=\phi$ above, we have $I(\phi)=0$ by \eqref{I}. Then applying \eqref{N}, it holds that
\begin{equation*}
\Phi\backslash \{0\}\subset\mathcal{N},
\end{equation*}
which, together with \eqref{555} and $\phi_0\in\Phi\backslash \{0\}$, shows that
\begin{equation*}
J(\phi_0)=d=\inf\limits_{\phi\in\mathcal{N}} J(\phi)\leq\inf\limits_{\phi\in\Phi\backslash \{0\}} J(\phi)\leq J(\phi_0)
\end{equation*}
and so $J(\phi_0)=\inf\limits_{\phi\in\Phi\backslash \{0\}} J(\phi)$. The proof is complete.
\end{proof}

\begin{remark}
Theorem \ref{jitaijie} implies that the potential well depth $d$ defined in \eqref{d} is achievable.
\end{remark}

%there exists a ground state solution on the Nehari manifold and. These are mathematically equivalent under the variational framework: the ground state solution attains the minimal energy $d$, and conversely, the achievability of $d$ implies the existence of such a solution.

\begin{proof}[Proof of Theorem \ref{ccc-22}]
Let $u=u(t)$ be a global solution to problem \eqref{1.1}.
From \eqref{Ju0}, we observe that $J(u(t))$ is non-increasing over $t\in[0,\infty)$. Hence, $J(u)\leq J(u_0)$ holds for all $t\in[0,\infty)$. We assert that $J(u)\geq 0$ for all $t\in[0,\infty)$. Suppose otherwise that there exists $t_0\in[0,\infty)$ such that $J(u(t_0))<0$. Combined with \eqref{J=I}, this implies $I(u(t_0))<0$. By Theorem \ref{baopo}, we conclude that $u$ blows up in finite time, contradicting the assumption that $u$ is a global solution. Therefore,
\begin{equation}\label{ff1-22}
0\leq J(u)\leq J(u_0),\quad \forall t\in[0,\infty).
\end{equation}
From \eqref{ff1-22} and the non-increasing property of $J(u(t))$ over $t\in[0,\infty)$, there exists a constant $C_0\in[0,J(u_0)]$ such that
\begin{equation}\label{ff2-22}
\lim\limits_{t\rightarrow\infty}J(u)=C_0.
\end{equation}
Letting $t\rightarrow\infty$ in \eqref{Ju0}, we obtain
\begin{equation}\label{ff3-22}
\int_0^\infty\int_\Omega\frac{u_\tau^2}{|x|^{2s}}dxd\tau+\int_0^\infty\left\|(-\Delta)^{\frac{s}{2}}u_\tau\right\|_2^2d\tau=J(u_0)-C_0\leq J(u_0),
\end{equation}
which implies that there exists an increasing sequence $\{t_k\}_{k=1}^\infty$ with $t_k\rightarrow\infty$ as $k\rightarrow\infty$ such that
\begin{equation*}
\lim_{k\rightarrow\infty}\left(\int\limits_\Omega\frac{(u'(t_k))^2}{|x|^{2s}}dx+\left\|(-\Delta)^{\frac{s}{2}}u'(t_k)\right\|_2^2\right)=0,
\end{equation*}
and consequently,
\begin{equation}\label{ff4-22}
\lim_{k\rightarrow\infty}\|u'(t_k)\|_{H_0^s}=\lim_{k\rightarrow\infty}\left\|(-\Delta)^{\frac{s}{2}}u'(t_k)\right\|_2=0.
\end{equation}

According to \eqref{J} and \eqref{weak}, for any $\phi\in H_0^s(\Omega)$, we have
\begin{equation}\label{ff5-22}
\begin{split}
&\quad\langle J'(u(t)),\phi\rangle\\
&=\frac{d}{d\tau}J(u(t)+\tau\phi)\bigg|_{\tau=0}\\
&=\left[\left((-\Delta)^{\frac{s}{2}}(u(t)+\tau\phi),(-\Delta)^{\frac{s}{2}}\phi\right)-\left(|u(t)+\tau\phi|^{p-2}(u(t)+\tau\phi),\phi\right)\right]\bigg|_{\tau=0}\\
&=\left((-\Delta)^{\frac{s}{2}}u(t),(-\Delta)^{\frac{s}{2}}\phi\right)-\left(|u(t)|^{p-2}u(t),\phi\right)\\
&=-\left(|x|^{-2s}u'(t),\phi\right)-\left((-\Delta)^{\frac{s}{2}}u'(t),(-\Delta)^{\frac{s}{2}}\phi\right),
\end{split}
\end{equation}
Combining \eqref{ff5-22}, \eqref{f111-22} and \eqref{fanshudy-22}, we derive
\begin{equation*}
\begin{split}
&\quad\|J'(u(t_k))\|_{H^{-s}(\Omega)}\\
&=\sup_{\|\phi\|_{H_0^s}\leq1}\big|\langle J'(u(t_k)),\phi\rangle\big|\\
&\leq\sup_{\|\phi\|_{H_0^s}\leq1}\left(\left\||x|^{-s}u'(t_k)\right\|_2\||x|^{-s}\phi\|_2+\left\|(-\Delta)^{\frac{s}{2}}u'(t_k)\right\|_2\left\|(-\Delta)^{\frac{s}{2}}\phi\right\|_2\right)\\
&\leq\sup_{\|\phi\|_{H_0^s}\leq1}\left\|(-\Delta)^{\frac{s}{2}}\phi\right\|_2\left(C_*\left\|(-\Delta)^{\frac{s}{2}}u'(t_k)\right\|_2+\left\|(-\Delta)^{\frac{s}{2}}u'(t_k)\right\|_2\right)\\
%&=\sup_{\|\phi\|_{H_0^s}\leq1}\|\phi\|_{H_0^s}[(C_*+1)\|u_t(t_k)\|_{H_0^s}]\\
&\leq(C_*+1)\|u'(t_k)\|_{H_0^s},
\end{split}
\end{equation*}
which, together with \eqref{ff4-22}, yields
\begin{equation}\label{ff6-22}
\lim_{k\rightarrow\infty}\|J'(u(t_k))\|_{H^{-s}(\Omega)}=0.
\end{equation}
Therefore, there exists a constant $c>0$ such that
\begin{equation}\label{fff8-22}
\begin{split}
|I(u(t_k))|&=|\langle J'(u(t_k)),u(t_k)\rangle|\\
&\leq\|J'(u(t_k))\|_{H^{-s}(\Omega)}\|u(t_k)\|_{H_0^s}\\
&\leq c\|u(t_k)\|_{H_0^s}.
\end{split}
\end{equation}
By \eqref{ff1-22}, \eqref{fff8-22} and \eqref{J=I}, we obtain
\begin{equation*}
\begin{split}
J(u_0)+\frac{c}{p}\|u(t_k)\|_{H_0^s}&\geq J(u(t_k))-\frac{1}{p}I(u(t_k))\\
&=\frac{p-2}{2p}\left\|(-\Delta)^{\frac{s}{2}}u(t_k)\right\|_2^2\\
&=\frac{p-2}{2p}\|u(t_k)\|_{H_0^s}^2,
\end{split}
\end{equation*}
which implies the existence of a constant $K_1>0$ such that
\begin{equation}\label{ff8-22}
\|u(t_k)\|_{H_0^s}\leq K_1, \quad k=1,2,\cdots.
\end{equation}
Since the embedding $H_0^s(\Omega)\hookrightarrow L^p(\Omega)$ is compact, there exists an increasing subsequence of $\{t_k\}_{k=1}^\infty$ (still denoted by $\{t_k\}_{k=1}^\infty$) and $u_*\in H_0^s(\Omega)$ such that, as $k\rightarrow\infty$ (we denote $u_k:=u(t_k)$):
\begin{align}
&u_k\rightharpoonup u_* \hbox{ weakly in } H_0^s(\Omega),\label{ff9-22}\\
&u_k\rightarrow u_* \hbox{ strongly in } L^p(\Omega).\label{ff10-22}
\end{align}

Next, we prove that $u_*$ is a solution to the stationary problem \eqref{2.1}.
From \eqref{ff10-22}, there exists a subsequence of $\{u_k\}_{k=1}^\infty$ (still denoted by $\{u_k\}_{k=1}^\infty$) and a function $v\in L^p(\Omega)$ such that, as $k\rightarrow\infty$,
\begin{equation}\label{ff21-22}
u_k(x)\rightarrow u_*(x) \hbox{ for a.e. } x\in\Omega,
\end{equation}
and for all $k$,
\begin{equation}\label{ff22-22}
|u_k(x)|\leq v(x) \hbox{ for a.e. } x\in\Omega.
\end{equation}
We show that for any $\phi\in H_0^s(\Omega)$, as $k\rightarrow\infty$,
\begin{equation}\label{ff23-22}
\left(|u_k|^{p-2}u_k,\phi\right)\rightarrow\left(|u_*|^{p-2}u_*,\phi\right).
\end{equation}
By \eqref{ff21-22}, as $k\rightarrow\infty$, we have
\begin{equation}\label{ff24-22}
|u_k(x)|^{p-2}u_k(x)\phi(x)\rightarrow |u_*(x)|^{p-2}u_*(x)\phi(x) \hbox{ for a.e. } x\in\Omega,
\end{equation}
and from \eqref{ff22-22}, it follows that for all $k$,
\begin{equation}\label{ff25-22}
\left||u_k(x)|^{p-2}u_k(x)\phi(x)\right|\leq \left||v(x)|^{p-2}v(x)\phi(x)\right| \hbox{ for a.e. } x\in\Omega.
\end{equation}
Applying H\"{o}lder's inequality and \eqref{C}, we derive
\begin{equation*}
\int_\Omega\left||v(x)|^{p-2}v(x)\phi(x)\right|dx\leq\|v(x)^{p-1}\|_{\frac{p}{p-1}}\|\phi(x)\|_p\leq C\|v(x)\|_p^{p-1}\|\phi(x)\|_{H_0^s},
\end{equation*}
Since $v\in L^p(\Omega)$ and $\phi\in H_0^s(\Omega)$, it follows that $\int_\Omega\left||v(x)|^{p-2}v(x)\phi(x)\right|dx<\infty$, i.e.,
\begin{equation*}
|v(x)|^{p-2}v(x)\phi(x)\in L^1(\Omega).
\end{equation*}
Combining \eqref{ff24-22}, \eqref{ff25-22} and invoking the Lebesgue's dominated convergence theorem, we conclude that \eqref{ff23-22} holds.
Letting $k\rightarrow \infty$ in
\begin{equation*}
\langle J'(u_k),\phi\rangle=\left((-\Delta)^{\frac{s}{2}} u_k,(-\Delta)^{\frac{s}{2}}\phi\right)-\left(|u_k|^{p-2}u_k,\phi\right),
\end{equation*}
and using \eqref{ff6-22}, \eqref{ff9-22} and \eqref{ff23-22}, we obtain
\begin{equation*}
0=\left((-\Delta)^{\frac{s}{2}} u_*,(-\Delta)^{\frac{s}{2}}\phi\right)-\left(|u_*|^{p-2}u_*,\phi\right),
\end{equation*}
which implies that $u_*$ is a solution to \eqref{2.1}.

Next, we prove that $\lim_{k\rightarrow\infty}\|u_k-u^*\|_{H_0^s}=0$.
From \eqref{J}, we derive
\begin{equation*}
\begin{split}
&\quad\langle J'(u_k),u_k-u_*\rangle\\
&=\frac{d}{d\tau}J(u_k+\tau (u_k-u_*))\bigg|_{\tau=0}\\
&=\big\{\left((-\Delta)^{\frac{s}{2}}[u_k+\tau(u_k-u_*)],(-\Delta)^{\frac{s}{2}}(u_k-u_*)\right)\\
&\quad-\left(|u_k+\tau(u_k-u_*)|^{p-2}[u_k+\tau(u_k-u_*)],u_k-u_*\right)\big\}\big|_{\tau=0}\\
&=\left((-\Delta)^{\frac{s}{2}}u_k,(-\Delta)^{\frac{s}{2}}(u_k-u_*)\right)-\left(|u_k|^{p-2}u_k,u_k-u_*\right),
\end{split}
\end{equation*}
Similarly,
\begin{equation}\label{ff12-22}
\begin{split}
&\quad\langle J'(u_*),u_k-u_*\rangle\\
&=((-\Delta)^{\frac{s}{2}}u_*,(-\Delta)^{\frac{s}{2}}(u_k-u_*))-\left(|u_*|^{p-2}u_*,u_k-u_*\right).
\end{split}
\end{equation}
Thus,
\begin{equation}\label{ff16-22}
\langle J'(u_k)-J'(u_*),u_k-u_*\rangle=\left\|(-\Delta)^{\frac{s}{2}}(u_k-u_*)\right\|_2^2-\Gamma,
\end{equation}
where $\Gamma:=\left(|u_k|^{p-2}u_k-|u_*|^{p-2}u_*,u_k-u_*\right)$.
By H\"{o}lder's inequality, \eqref{C} and \eqref{ff8-22}, we obtain
\begin{equation*}
\begin{split}
|\Gamma|%&\leq|\left(u_*^p-(u(t_k))^p,u(t_k)-u_*\right)|\\
&\leq\|u_*^{p-1}-(u_k)^{p-1}\|_{\frac{p}{p-1}}\|u_k-u_*\|_p\\
&\leq\left(\|u_*\|_p^{p-1}+\|u_k\|_p^{p-1}\right)\|u_k-u_*\|_p\\
&\leq\left(C^{p-1}\|u_*\|_{H_0^s}^{p-1}+C^{p-1}\|u_k\|_{H_0^s}^{p-1}\right)\|u_k-u_*\|_p\\
&\leq\left(C^{p-1}\|u_*\|_{H_0^s}^{p-1}+C^{p-1}K_1^{p-1}\right)\|u_k-u_*\|_p,
\end{split}
\end{equation*}
Combining this with \eqref{ff10-22}, we conclude
\begin{equation}\label{ff11-22}
\lim_{k\rightarrow\infty}|\Gamma|=0.
\end{equation}
By \eqref{ff9-22}, \eqref{ff10-22} and \eqref{ff12-22}, we deduce
\begin{equation}\label{ff13-22}
\lim_{k\rightarrow\infty}\langle J'(u_*),u_k-u_*\rangle=0.
\end{equation}
From \eqref{ff8-22}, we have
\begin{equation*}
\begin{split}
\big|\langle J'(u_k),u_k-u_*\rangle\big|&\leq\|J'(u_k)\|_{H^{-s}(\Omega)}(\|u_k\|_{H_0^s}+\|u_*\|_{H_0^s})\\
&\leq\|J'(u_k)\|_{H^{-s}(\Omega)}(K_1+\|u_*\|_{H_0^s}),
\end{split}
\end{equation*}
which, together with \eqref{ff6-22}, yields
\begin{equation}\label{ff15-22}
\lim_{k\rightarrow\infty}\big|\langle J'(u_k),u_k-u_*\rangle\big|=0.
\end{equation}
It follows from \eqref{ff16-22}, \eqref{ff11-22}, \eqref{ff13-22} and \eqref{ff15-22} that
\begin{equation*}
\begin{split}
\lim_{k\rightarrow\infty}\left\|(-\Delta)^{\frac{s}{2}}(u_k-u_*)\right\|_2^2
=\lim_{k\rightarrow\infty}\left[\langle J'(u_k)-J'(u_*),u_k-u_*\rangle+\Gamma\right]
=0,
\end{split}
\end{equation*}
which completes the proof.
\end{proof}

\begin{remark}
From Theorem \ref{1111} we notice that the global solution decays to $0$ at the exponential rate as $t\rightarrow\infty$ under some suitable conditions, while Theorem \ref{ccc-22} indicates that the general global solution of \eqref{1.1} strongly converge to the solution of \eqref{2.1} as $t\rightarrow \infty$.
\end{remark}

\section*{Compliance with Ethical Standards}
{\bf Conflict of interest} The authors declare that they have no conflict of interest.

\section*{Acknowledgement}
The authors would like to express deep gratitude to Professor Marcel Clerc for his very helpful suggestions. This work was supported by the National Natural Science Foundation of China (12171339 and 12471296).

%\bibliographystyle{plain}
%\bibliography{a}

%\end{CJK}

\end{document}